\documentclass[12pt,reqno]{amsart}

\usepackage{soul}
\usepackage{amsthm}
\usepackage{amssymb}
\usepackage{latexsym}
\usepackage{multicol,multirow}
\usepackage{verbatim,enumerate}
\usepackage{accents}
\usepackage[usenames]{color}
\usepackage[colorlinks=true, linkcolor=blue, citecolor=green, urlcolor=blue]{hyperref}
\usepackage[nameinlink,noabbrev,capitalize]{cleveref}
\usepackage{nicematrix,makecell}
\usepackage{amsmath, amscd}
\usepackage{soul}
\usepackage{pifont}
\usepackage{dsfont}
\usepackage{dirtytalk,caption,subcaption}
\usepackage{tikz,setspace}
\usepackage[T1]{fontenc}
\usetikzlibrary{matrix,arrows,decorations.pathmorphing}
\usetikzlibrary{positioning}

\advance\textwidth by 1.2in \advance\oddsidemargin by -.6in \advance\evensidemargin by -.6in

\theoremstyle{plain}
\newtheorem{prop}{Proposition}
\newtheorem{thm}{Theorem}

\newtheorem{cor}{Corollary}%[section]

\theoremstyle{definition}
\newtheorem{example}{Example}
\newtheorem{defn}{Definition}

\theoremstyle{remark}
\newtheorem{rem}{Remark}

\newenvironment{pf}{\proof}{\endproof}
\newcounter{cnt}
 \makeatletter
\def\mydggeometry{\makeatletter\dg@YGRID=1\dg@XGRID=20\unitlength=0.003pt\makeatother}
\makeatother \theoremstyle{remark}

\numberwithin{equation}{section}
\makeatletter
\def\section{\def\@secnumfont{\mdseries}\@startsection{section}{1}%
  \z@{.7\linespacing\@plus\linespacing}{.5\linespacing}%
  {\normalfont\scshape\centering}}
\def\subsection{\def\@secnumfont{\bfseries}\@startsection{subsection}{2}%
  {\parindent}{.5\linespacing\@plus.7\linespacing}{-.5em}%
  {\normalfont\bfseries}}
\makeatother

\begin{document}

\title[]{On the characterization of chordal graphs using Horn hypergeometric series}
\author{Dipnit Biswas}\address{Department of Mathematics, Indian Institute of Science, Bangalore 560012, India}
\email{dipnitbiswas@iisc.ac.in}
\thanks{}
\author{Irfan Habib}\address{Department of Mathematics, Indian Institute of Science, Bangalore 560012, India}
\email{irfanhabib@iisc.ac.in}
\thanks{}

\author{R. Venkatesh}
\address{Department of Mathematics, Indian Institute of Science, Bangalore 560012, India}
\address{Visiting Faculty, Department of Mathematics, Indian Institute of Technology Madras, Chennai 600036, India}
\email{rvenkat@iisc.ac.in}
%\thanks{R.V. was partially supported by the Infosys Young Investigator Award grant.}

\subjclass[2010]{}
\begin{abstract}
Radchenko and Villegas characterized the chordal graphs by the inverse of their independence polynomials being Horn hypergeometric series in Radchenko et al. in 2021.
In this paper, we reprove their result using some elementary combinatorial methods. Our proof is different from their proof, and it is based on the connection between the inverse of the multi-variate independence polynomials and the multi-colored chromatic polynomials of graphs, established by Arunkumar et al. in 2018.
\end{abstract}

\maketitle
%%%%%%%%%%%%%%%%%%%%%%%%%%%%%%%%%%%%%%%%%%%%%%%%%%%%%%%%%%%%%%%%%%%%%%%%%%%%%%%%%%%%%%%%%%%%%%%%%%%%%%%%%%%%%%%%%%%%%%%%%%%%%%%%%%%%%%%%

%%%%%%%%%%%%%%%%%%%%%%%%%%%%%%%%%%%%%%%%%%%%%%%%%%%%%%%%%%%%%%%%%%%%%%%%%%%%%%%%%%%%%%%%%%%%%%%%%%%%%%%%%%%%%%%%%%%%%%%%%%%%%%%%%%%%%%%%
\section{Introduction}
The main motivation of this paper is to reprove a result of Radchenko and Villegas, which characterizes the chordal graphs by the inverse of their independence polynomials being Horn hypergeometric series, see \cite[Theorem 2.2]{RadchenkoVilleas2021}, using some elementary ideas coming from graph colorings. First, we state their result and then explain our strategy to approach their result. 
We consider only \textit{finite simple graphs} in this paper, which means our graphs have only finitely many vertices and edges and no multiple edges and loops. 
Let $\mathcal{G}$ be a simple graph with a vertex set $I = \{1, \ldots, n\}$. A subset of $I$ is said to be stable or independent if the subgraph spanned by that subset has no edges. 
We work over the field of rational numbers $\mathbb{Q}$ throughout this article and let $R = \mathbb{Q}[[x_i : i\in I]]$ be the ring of formal power series generated by the commuting variables $(x_i : i\in I)$ over $\mathbb{Q}$.
%{\color{red} Maybe: Let $\mathbb{Q}$ be a field of characteristic $0$ and let $R = \mathbb{Q}[[x_i : i\in I]]$ be the ring of formal power series generated by the commuting variables $x_i, i\in I$ over $\mathbb{Q}$.}
The \textit{multi-variate independence polynomial} of $\mathcal{G}$ over $\mathbb{Q}$ is defined to be the polynomial 
$$ I(\mathcal{G}, \mathbf x) = \sum\limits_{S\in \mathcal{I}(\mathcal{G})} \left(\prod_{i\in S} x_i\right) \in R,$$
where  $\mathcal{I}(\mathcal{G})$ is the set of all 
\textit{independent subsets} of $\mathcal G$. Note that $\emptyset$ and $\{i\}\in \mathcal{I}(\mathcal{G})$ for all $i\in I.$
Since the power series $I(\mathcal{G}, \mathbf x)$ has constant term $1$, we have $I(\mathcal{G}, \mathbf x)^{-q}\in R$ for any positive integer $q\ge 1.$ 
For a given  $I-$tuple of integers $\mathbf{m} =(m_i : i\in I)$, the support of $\mathbf{m}$ is defined by
$\mathrm{supp}(\mathbf{m}) = \{i\in I : m_i\neq 0\}$, and we say $\mathbf{m}\ge 0$ if $m_i\ge 0$ for all $i\in I.$
A power series $F(\mathbf{x})=\sum_{\mathbf{m} \geq 0} c_{\mathbf{m}}\mathbf{x}^{\mathbf{m}}\in R$ is called $\textit{Horn hypergeometric}$ if
\begin{enumerate}
    \item 
 $c_{\mathbf{m}}\neq 0$ for all $\mathbf{m}\geq 0$  and
 \medskip
\item 
$\frac{c_{\mathbf{m}+\mathbf{e}_i}}{c_{\mathbf{m}}}$
is a rational function of $\{m_j : j\in \mathrm{supp}(\mathbf{m})\}$ for each $i\in \mathrm{supp}(\mathbf{m})$, where $\mathbf{e}_i := (\delta_{ji} : j\in I)$ for $i\in I$,
\end{enumerate}
%The notations $\mathbf{m}\ge 0$ and $|\mathrm{supp}(\mathbf{m})|$ are defined in Section \ref{basicnotations}.

We urge interested
readers to refer to \cite{RadchenkoVilleas2021} for more details and motivation to study the Horn hypergeometric series; we will not discuss it in this paper.
A finite simple graph $\mathcal{G}$ is said to be \textit{chordal} if it has no induced subgraph isomorphic to the cycle graph $C_m$ with $m \geq 4$. 
By an \textit{induced subgraph} defined by a subset $S$ of vertices 
of $\mathcal{G}$, we mean the subgraph  $\mathcal{G}(S)\subseteq \mathcal{G}$ 
obtained by deleting from $\mathcal{G}$ the vertices not in $S$  and all the edges incident to at least one vertex not in $S$. For $n\ge 3$, the cycle graph $C_n$ consists 
of vertices $\{1,\ldots,n\}$ with an edge joining~$i$ with $i+1$, where the indices are read modulo~$n$.
\begin{figure}[ht]
 \begin{subfigure}{0.49\textwidth}

    \centering
    \begin{tikzpicture}
    \tikzstyle{B}=[circle,draw=black!80,fill=black!80,thick]
    %\tikzstyle{W}=[circle,draw=white!60,fill=white!60,thick]
    \node[B] (1) at (0,0) [label=below:$n-1$]{};
    \node[B] (2) at (1.5,0) [label=below:$3$]{};
    \node[B] (3) at (-0.8,1.5) [label=left:$n$]{};
    \node[B] (4) at (2.3,1.5) [label=right:$2$]{};
    \node[B] (5) at (0.75,2.7) [label=right:$1$]{};
    \path[-] (2) edge node[left]{} (4);
    \path[-] (1) edge node[left] {} (3);
    \path[-] (4) edge node[left] {} (5);
    \path[-] (3) edge node[left] {} (5);
    %\draw[dotted] (3) -- (4);
    \path (1) -- node[auto=false]{\ldots} (2);
\end{tikzpicture}
    \caption{Cycle $C_n$}
    \label{C_n}
        \end{subfigure}
    \begin{subfigure}{0.49\textwidth}

\centering
\begin{tikzpicture}
    \tikzstyle{B}=[circle,draw=black!80,fill=black!80,thick]
    \node[B] (1) at (0,0) [label=below:$1$]{};
    \node[B] (2) at (2,0) [label=below:$2$]{};
    \node[B] (3) at (3.5,0) [label=below:$n-1$]{};
    \node[B] (4) at (5.5,0) [label=below:$n$]{};
    \path[-] (1) edge node[left]{} (2);
    \path[-] (3) edge node[left]{} (4);
    \path (2) -- node[auto=false]{\ldots} (3);
\end{tikzpicture}
    \caption{Line graph $L_n$}
    \label{L_n}
    \end{subfigure}
    \caption{}
\end{figure}

\noindent
For example, the line graphs $L_n, n\ge 1$ are chordal, but the cycle graphs $C_n, n\ge 4$ are not chordal. Let us keep
the line graph $L_4$ and the cycle graph $C_4$ as our running examples to explain all our results. One can easily compute the multi-variate independence polynomials of these graphs. 
\begin{align*}
I(L_4, \bold x)&= 1+x_1+x_2+x_3+x_4+x_1x_3+x_2x_4+x_1x_4,\\
I(C_4, \bold x)&= 1+x_1+x_2+x_3+x_4+x_1x_3+x_2x_4.\\
\end{align*}
\noindent
The following characterization of chordal graphs is proved in \cite[Theorem 2.2]{RadchenkoVilleas2021}.
\begin{thm}(\cite[Theorem 2.2]{RadchenkoVilleas2021})\label{mainthmintro}
 Let $\mathcal{G}$ be a finite simple graph. Then, the
following statements are equivalent.
\begin{enumerate}
    \item 
 The graph $\mathcal{G}$ is a chordal graph.
 \item The power series $I(\mathcal{G}, \mathbf x)^{-q}$  is Horn hypergeometric
         for all $q\ge 1$.
       \item The power series $I(\mathcal{G}, \mathbf x)^{-1}$  is Horn hypergeometric. \qed
\end{enumerate} 
\end{thm}

This paper reproves Radchenko and Villegas's result using more elementary combinatorial
methods. We now explain our strategy. Our proof is motivated by the connection between the generalized chromatic polynomials of $\mathcal{G}$ and the $q$--th power of independence polynomial of $\mathcal{G}$ which was established in \cite{ADV2018}. Fix $\mathbf{m} = (m_i : i \in I)$ a tuple of non-negative integers, set $|\mathbf{m}| = \sum_{i\in I} m_i$. Given $q$ colors,
a proper vertex  multi coloring of $\mathcal{G}$ associated to $\mathbf{m}$, using given at most $q$-colors, is an assignment of colors from these given $q$-colors to the vertices of $\mathcal{G}$ in which each vertex $i\in I$
receives exactly $m_i$ colors such that two adjacent vertices receive
 disjoint colors. For a given $q\in \mathbb{N}$, the function that counts the number of distinct
proper vertex multicolorings of $\mathcal{G}$ that can be made using given $q$-colors is a polynomial in $q$, called 
the generalized chromatic polynomial of $\mathcal{G}$  associated to $\mathbf{m}$, it is denoted as $\pi_{\mathcal{G}}^{\mathbf{m}}(q)$ (see Section \ref{multichormaticdef} for precise definition). Then we have (see \cite[Propostion 3.7]{ADV2018}, also \cite[Remark 3.4]{Bernardi})
\begin{thm}\label{genfunintro}
Let $\mathcal{G}$ be a finite simple graph. For $q\in \mathbb{Z}$, we have the following identity holds in $R$
\[ I(\mathcal{G}, \mathbf{x})^{q} = \sum\limits_{\mathbf{m}\ge 0} \pi_{\mathcal{G}}^{\mathbf{m}}(q) \mathbf{x}^{\mathbf{m}}.
\] \qed
\end{thm}

\medskip
Theorem \ref{genfunintro} gives a very explicit way of calculating the $q-$th power of $I(\mathcal{G}, \mathbf{x})$ in terms of generalized chromatic polynomials of $\mathcal{G}$. 
On the other hand, one can efficiently compute these generalized chromatic polynomials for graphs with the perfect elimination ordering on their vertices.
An ordering $\leq $ on  the vertices $I$ of $\mathcal{G}$ is called 
a \textit{perfect elimination ordering} if it is a total-ordering and 
for each $k\in I$ the subgraph $\mathcal{G}_k$ of $\mathcal{G}$ induced by the set of vertices from  
$\{i\in I : i\le k\}$ that are adjacent with $k$ is a clique, i.e., complete subgraph (here, by convention, we assume that $k$ is adjacent to $k$). We call $\mathcal{G}$ a PEO graph if its vertices have a perfect elimination ordering. For example, the usual ordering on the vertices of the Line graph $L_n$ is a perfect elimination ordering. Not every graph has a perfect elimination ordering, for example the cycle graph $C_n$ is a PEO graph if and only if $n=3.$
It is a remarkable fact that a finite simple graph $\mathcal{G}$ is chordal if and only if it is a PEO graph \cite[Chapter 4, Section 1]{Golumbic}. 
It is easy to compute the chromatic polynomials of chordal graphs if we fix a perfect elimination ordering on their vertices. More precisely, we have
\begin{thm}\label{poegraphintro}
Let $\mathcal{G}$ be a chordal graph with the vertex set $I$ and let $I=\{1, \ldots, n\}$ be a perfect elimination ordering. 
Then, for any $\mathbf{m}\ge 0$, we have
 \[
\pi_{\mathcal{G}}^{\mathbf{m}}(-q) =(-1)^{|\mathbf{m}|}\prod\limits_{r=1}^{n} \binom{q-1+a_r(\mathbf{m})}{m_{r}}
\]
where $a_r(\mathbf{m})=m_{r}+\sum\limits_{\substack{1\le s <r\\ s\sim r}} m_{s}$ for $1\le r \le n$ (here $a\sim b$ means the vertices $a$ and $b$
are connected by an edge in $\mathcal{G}$). \qed
\end{thm}
\noindent
For example, we have \[
\pi_{L_4}^{\mathbf{m}}(-q) =(-1)^{|\mathbf{m}|} \binom{q+m_1-1}{m_1} \binom{q+m_1+m_2-1}{m_{2}} \binom{q+m_2+m_3-1)}{m_{3}} \binom{q+m_3+m_4-1}{m_{4}}.
\]

\medskip\noindent
Note that we recover one of the main results of \cite[Corollary 3.2]{RadchenkoVilleas2021} from Theorem \ref{genfunintro} and \ref{poegraphintro}, which was proved in \cite{RadchenkoVilleas2021} by entirely different techniques. 
\begin{cor}\label{cor3.2}
Let $\mathcal{G}$ be a chordal graph with the vertex set $I$ and let $I=\{1, \ldots, n\}$ be a perfect elimination ordering. 
Then, for any $q\ge 1$,  we have
 \[  I(\mathcal{G}, \mathbf{x})^{-q} = \sum\limits_{\mathbf{m}\ge 0} (-1)^{|\mathbf{m}|}\prod\limits_{r=1}^{n} \binom{q-1+a_r(\mathbf{m})}{m_{r}} \mathbf{x}^{\mathbf{m}}.
\]
where $a_r(\mathbf{m})$ defined as above. In particular $I(\mathcal{G}, \mathbf{x})^{-q}$ is a Horn hypergeometric series for all $q\ge 1.$ \qed
\end{cor}
Radchenko and Villegas used ideas coming from the representation theory of quantum affine algebras; they inverted the Nahm system of equations arising from the adjacency matrix of the given graph and, using that, wrote down the formula for the coefficients of the inverse of the multi-variate independence polynomial of chordal graphs and cycle graphs. 
Our method has some advantages as we have that Theorem \ref{genfunintro} holds for any finite simple graph $\mathcal{G}$. 
So, we can compute the generalized chromatic polynomials for even some families of non-chordal graphs; in particular, we could compute this for cycle graphs, which was already done in Read's paper \cite{Read}. It is important to note that the authors computed only  $I(C_n, \mathbf x)^{-1}, n\ge 4$  in \cite{RadchenkoVilleas2021} using their techniques, but
the general formula for $I(C_n, \mathbf x)^{-q}$  for $n\ge 4$ can be easily written down
using our method and Read's result. %We can use this information to prove our main theorem without difficulties, as in \cite{RadchenkoVilleas2021}.

\medskip
Indeed, Read calculated the chromatic polynomials of the clan circuit graphs using mathematical induction and the fundamental recurrence relation of chromatic polynomials called the deletion–contraction recurrence relation. Here the \textit{clan circuit of graph} length $n\ge 1$ associated with an $n-$tuple of non-negative integers $(m_1, \ldots, m_n)$ is defined as follows:
for each vertex $1\le i \le n$ of $C_n$, we take a complete graph of size $m_i$  and 
join all the vertices of $i$th complete graph of size $m_i$ and $j$th complete graph of size $m_j$ if there is an edge between the vertices $i$ and $j$ in $C_n.$ Here 
the complete graph of size $r$ is the graph that has $\{1, \ldots, r\}$ as vertices and any two distinct vertices $1\le i \neq j\le r$ 
have an edge between them.
For example $n = 3$ and $m_1 = 2, m_2 = 3$ and $m_3 =1$, we have $C_3(2, 3, 1)$ as shown in Figure \ref{clangraph}.
\begin{figure}[h]
\centering
\begin{subfigure}{0.49\textwidth}
\centering
        \begin{tikzpicture}
    \node (1) at (-4,0) [circle, fill=black,thick, label=below:1] {};
    \node (2) at (-2,0) [circle, fill=black,thick, label=below:2] {};
    \node (3) at (3,0) [circle, fill=black,thick, label=below:A] {};
    \node (4) at (-1,2) [circle, fill=black,thick, label=left:a] {};
    \node (5) at (1,2) [circle, fill=black,thick, label=right:b] {};
    \node (6) at (0,3) [circle, fill=black,thick, label=above:c] {};
    
    \path[-] (1) edge node[left]{} (2);
    \path[-,bend right] (1) edge node[left]{} (3);
    \path[-] (2) edge node[left]{} (3);
    \path[-] (4) edge node[left]{} (5);
    \path[-] (4) edge node[left]{} (6);
    \path[-] (5) edge node[left]{} (6);
    \path[-,bend left] (1) edge node[left]{} (6);
    \path[-] (1) edge node[left]{} (4);
    \path[-] (1) edge node[left]{} (5);
    \path[-,bend left=120] (2) edge node[left]{} (6);
    \path[-] (2) edge node[left]{} (4);
    \path[-] (2) edge node[left]{} (5);
    \path[-,bend right=90] (3) edge node[left]{} (6);
    \path[-] (3) edge node[left]{} (4);
    \path[-] (3) edge node[left]{} (5);
\end{tikzpicture}
 \caption{The clan circuit graph $C_3(2, 3, 1)$}
    \label{clangraph}
\end{subfigure}
\caption{}
\end{figure}

\medskip
However, we have a close relationship between the chromatic polynomials of the clan circuit graphs and 
generalized chromatic polynomials of cycle graphs; see the section \ref{clanconnection} and Proposition \ref{keyprop} for more details. Using 
Proposition \ref{keyprop} and \cite[Theorem 3]{Read}, we could write down the formula for the generalized chromatic polynomials of cycle graphs.
\begin{thm}(\cite[Theorem 3]{Read})\label{readthm}
    For the cycle graph $C_n$, $n\ge 3$, we have the following formula
$$   \pi_{C_n}^{\mathbf{m}}(q)= \frac{1}{\prod_{i=1}^n m_i!}  \left(\prod_{i=1}^{n}(q)_{(m_i+m_{i+1})}\right)\times
    \left(\sum\limits_{k=0}^n\Big((-1)^{kn}v_k(q)\prod_{i=1}^n\frac{(m_i)_{(k)}}{(q)_{({m_i+k})}}\Big)\right)
$$
where $m_{n+1}=m_1$ 
and $v_k(q)=\binom{q}{k}-\binom{q}{k-1}$ and $(q)_k=q(q-1)(q-2)\cdots(q-k+1)$.  \qed
\end{thm}
\noindent
For example, we have \[
\pi_{C_4}^{\mathbf{m}}(-q) = \frac{1}{\prod_{i=1}^4 m_i!}  \left(\prod_{i=1}^{4}(q)_{(m_i+m_{i+1})}\right)\times
    \left(\sum\limits_{k=0}^4v_k(q)\prod_{i=1}^4\frac{(m_i)_{(k)}}{(q)_{({m_i+k})}}\right).
\]
We can give our proof of Theorem \ref{mainthmintro}
using the stated results. A brief outline of our proof is given below.
\subsection{Outline of the proof}
The implication $(1)$ implies $(2)$ follows from the Corollary \ref{cor3.2} and the detailed proof is given in corollary \ref{keyresultpoegraph} (Section \ref{1implies2}).
The implication $(2)$ implies $(3)$ is obvious, and to prove $(3)$ implies $(1)$, we first observe that it is enough to prove  $I(C_m, \mathbf{x})^{-1}$ is 
Horn hypergeometric if and only if $m\le 3.$ The series 
$I(C_m, \mathbf x)^{-1} = \sum\limits_{\mathbf{m}\ge 0} c_{\mathbf{m}} \bold x^{\mathbf{m}}$ is Horn hypergeometric immediately implies that
 the sum of (signed) diagonal terms $$H_m(t) = \sum\limits_{\mathbf{a}= (a,\ldots, a)} (-1)^{ma} c_{\bold a} t^{a}\in \mathbb{Q}[[t]].$$
 is also a hypergeometric series. Set $S(m, a):= (-1)^{ma} c_{\bold a}$, and we have a closed formula for $S(m, a)$ deduced from Theorem \ref{readthm}; see corollary \ref{keyresultcycle}.  The numbers $S(m, a)$ are known to be the de Bruijn numbers in the literature; see \cite[Page 72, Section 4.7]{Bruijn}. 
The asymptotic behavior of $S(m, a)$ for fixed $m$ and large $a$ has been computed in \cite{Bruijn}:
	$$\frac{S(m, a+1)}{S(m, a)}\rightarrow (2\cos(\pi/2m))^{2m}, \qquad a \rightarrow\infty.$$
Since $H_m(t)$ is Horn
hypergeometric, we must have $(2\cos(\pi/2m))^{2m}\in \mathbb{Q}$, and this implies immediately that $m\le 3$ as in \cite[Proposition 6.3]{RadchenkoVilleas2021}.
More details of the proof can be found in the section \ref{converseproof}.

\section{Characterization of chordal graphs using hypergeometic series}\label{prem}

\subsection{Preliminaries}\label{basicnotations}
We denote by $\mathbb{N}$, $\mathbb{Z}_{\ge 0}$,  $\mathbb{Z}_{\le 0}$, and $\mathbb{Z}$ the set of positive integers, non-negative integers, non-positive integers, and integers, respectively. 
%Let $I$ be a countable indexing set, identify $I = \{1, \ldots, n\}$ or $\mathbb{N}.$
%Let $\mathbb{Q}$ be a field of characteristic $0$ and the ring of formal power series generated by the commuting variables $x_i, i\in I$ over $\mathbb{Q}$ is denoted as $R = \mathbb{Q}[[x_i : i\in I]]$.
%The ring of formal series in one variable $t$ over a field $\mathbb{Q}$ is denoted as $\mathbb{Q}[[t]].$ 
For a set $S$, denote by $|S|$ the cardinality of $S$ and by $P(S)$ the power set of $S.$
%Denote by $\mathbb{Z}_{\ge 0}^I$ the set of $I$--tuple of non-negative integers. For a given $\mathbf{m}: =(m_i : i\in I)\in \mathbb{Z}_{\ge 0}^I$, the support of $\mathbf{m}$ is denoted by
%$\mathrm{supp}(\mathbf{m}) = \{i\in I : m_i\neq 0\}$.
%Given $\mathbf{m}=(m_i : i\in I)\in \mathbb{Z}_{\ge 0}^I$ with finite support, i.e., $|\mathrm{supp}(\mathbf{m})|<\infty$, we define
%$|\mathbf{m}|=\sum_{i\in I}m_i.$ 
For a variable $q$, we define  ${q\choose k} = \frac{q(q-1)\cdots (q-(k-1))}{k!}.$
Note that we have ${-q\choose k} = 
 (-1)^k{q+k-1\choose k}.$

\subsection{}
Let $\mathcal{G}$ be a finite simple graph with a vertex set $I = \{1, \dots, n\}$ and an edge set $E$. %Note that $\mathcal{G}$ has no multiple edges and loops.
For $i, j\in I$, the edge between $i$ and $j$ is denoted as $e(i, j).$ Note that $e(i, j)= e(j, i)$ for all $i, j\in I.$ We now collect the basic definitions. 
\begin{defn}
\begin{enumerate}
    \item The $\textit{induced subgraph}$ defined by a subset $S$ of vertices of $\mathcal{G}$ is denoted as $\mathcal{G}(S)$ and is obtained by deleting from $\mathcal{G}$ the vertices not in $S$  and all the edges incident to at least one vertex not in $S$.
\item 
A subset $S\subseteq I$ is said to be \textit{independent or stable} if $e(i, j)\notin E$ for all $i, j\in S.$ Denote by  $\mathcal{I}(\mathcal{G})$  the set of all 
\textit{independent subsets} of $\mathcal G$. Note that we have the empty set $\emptyset $ and $\{i\}$ are elements of $\mathcal{I}(\mathcal{G})$ for all $i\in I.$
    \item 
The \textit{multi-variate independence polynomial} of $\mathcal{G}$ over $\mathbb{Q}$ is defined to be the polynomial
$$ I(\mathcal{G}, \mathbf x) = \sum\limits_{S\in \mathcal{I}(\mathcal{G})} \left(\prod_{i\in S} x_i\right).$$
\end{enumerate}
\end{defn}

\subsection{Generalized Chromatic Polynomials}\label{multichormaticdef}

 Let $\mathbf{m} = (m_i : i \in I)$ be a tuple of non-negative integers.
\begin{enumerate}
\item
We call a map\ $\tau^{\mathbf{m}}_{\mathcal{G}} : I \to P\left(\{1, \ldots, q\}\right)$ a $\textit{proper vertex multi-coloring}$ of $\mathcal{G}$ associated to $\mathbf{m}$ using at most $q$-colors if the following conditions are satisfied:
\begin{enumerate}
    \item for all $i \in I$, we have $|\tau^{\mathbf{m}}_{\mathcal{G}}(i)| = m_i$;
    \item for all $i, j \in I$ such that $e(i, j) \in E$, we have $\tau^{\mathbf{m}}_{\mathcal{G}}(i) \cap \tau^{\mathbf{m}}_{\mathcal{G}}(j) = \emptyset$.
\end{enumerate}

\medskip
\item
We say that two given ${}_{1}\tau^{\mathbf{m}}_{\mathcal{G}}$ and ${}_{2}\tau^{\mathbf{m}}_{\mathcal{G}}$  proper vertex multi-coloring of $\mathcal{G}$ associated to $\mathbf{m}$ using at most $q$-colors are \textit{distinct} if ${}_{1}\tau^{\mathbf{m}}_{\mathcal{G}}(i)\neq {}_{2}\tau^{\mathbf{m}}_{\mathcal{G}}(i)$ for some $i\in I.$

\medskip
\item
For a given $q\in \mathbb{N}$,
the number of proper vertex multi-coloring of $\mathcal{G}$ associated to $\mathbf{m}$ using at most $q$ colors is well-known to be a polynomial in $q$, called the $\textit{generalized chromatic}$ $\textit{polynomial}$ of $\mathcal{G}$ associated to $\mathbf{m}$, and it is denoted as
$\pi_{\mathcal{G}}^{\mathbf{m}}(q)$. 
\end{enumerate}
The generalized chromatic polynomial has the following well-known description.
As before, we take $\mathbf{m} = (m_i : i \in I)$ as a tuple of non-negative integers with finite support.
We denote by $P_k(\mathbf{m}, \mathcal{G})$ the set of all ordered $k$-tuples $(P_1, \ldots, P_k)$ such that:
\begin{enumerate}
    \item each $P_i$ is a non-empty independent subset of $I$, i.e., no two vertices have an edge between them; and
    \item the disjoint union of $P_1, \ldots, P_k$ is equal to the multiset $\{ i, \ldots, i : i \in \mathrm{supp}(\mathbf{m}) \}$, where $i$ appears exactly $m_i$ number of times for each $i\in \mathrm{supp}(\mathbf{m})$.
\end{enumerate}
The following is well-known; see \cite[Section 3.3]{ADV2018}.

\begin{prop}
For each $q\in \mathbb{N}$, we have
\[
\pi_{\mathcal{G}}^{\mathbf{m}}(q) = \sum_{k \geq 0} |P_k(\mathbf{m}, \mathcal{G})| \binom{q}{k}.
\] 
In particular, $\pi_{\mathcal{G}}^{\mathbf{m}}(q)$ depends only on $\mathbf{m}$ and $\mathcal{G}$ and does not depend on $q$ and it is a polynomial in $q$.  \qed
\end{prop}
\noindent
\begin{rem}
If $m_i = 1$ for all $i\in I$, then the polynomial $\pi_{\mathcal{G}}^{\mathbf{m}}(q)$ is the ordinary chromatic polynomial of $\mathcal{G}$ which is well-studied in the literature (see \cite{Birkhoff, Dong}). 
The definition of the generalized chromatic polynomial depends on the positive integer $q$; however, it
makes sense to consider the evaluation at any integer as it is a polynomial in $q$. A famous result of Stanley says that the evaluation at $q=-1$ of the chromatic polynomial counts the number of acyclic orientations of the graph $\mathcal{G}$ (see \cite{Stanley} for more details). 
\end{rem}

\subsection{}\label{clanconnection}
We have a close relationship between ordinary chromatic polynomials and multi-colored chromatic polynomials. To see this connection, 
we need the following definition of join of $\mathcal{G}$. 
Let  $\mathbf{m} = (m_i : i\in I)\in \mathbb{Z}_{\ge 0}^I$. The graph
$\mathcal{G}(\mathbf{m})-$the \textit{join of $\mathcal{G}$} with respect to $\mathbf{m}$ is defined as follows: 
the vertices of $\mathcal{G}(\mathbf{m})$ are the disjoint union of $\{i^1, \ldots, i^{m_i}\}$, $i\in \mathrm{supp}(\mathbf{m})$ and the edges
of $\mathcal{G}(\mathbf{m})$ are $e(i^r, i^s)$ for all $1\le r\neq s\le m_i$ and
$e(i^r, j^s)$ for $e(i, j)\in E(\mathcal{G})$ and $1\le r \le m_i$, $1\le s \le m_j$.
That is, for each vertex $i\in \mathrm{supp}(\mathbf{m})$, we take a clique (or complete subgraph) of size $m_i$ and 
join the all the vertices of $i$th clique and $j$th clique if there is an edge between the vertices $i$ and $j$ in the original graph $\mathcal{G}.$
The join graph has many different names in the literature, for example Read used clan graph in his paper. 
The following is well-known (see for example \cite[Remark 3.3]{ADV2018})
\begin{prop}\label{keyprop}
For $\mathbf{m} = (m_i : i\in I)\in \mathbb{Z}_{\ge 0}^I$, we have
$$\pi_{\mathcal{G}}^{\mathbf{m}}(q) =  \frac{\pi_{\mathcal{G}(\mathbf{m})}(q)}{\prod\limits_{i=1}^n m_i!},$$
where $\pi_{\mathcal{G}(\mathbf{m})}(q)$ is the ordinary chromatic polynomial of $\mathcal{G}(\mathbf{m}).$\qed
\end{prop}

\subsection{}
Using the first counting principles, we can efficiently compute the generalized chromatic polynomials for many interesting families of graphs.
For example, for the star graph $S_n$, we have \[
\pi_{S_n}^{\mathbf{m}}(q) = \binom{q}{m_1}\binom{q-m_1}{m_2}\cdots \binom{q-m_{1}}{m_n}. 
\]
This can be computed as follows: label the vertices of the star graph as in the figure \ref{stargraph}.
Let us say $q$ colors are given to color the vertices of $S_n$. As we do not have any restrictions in coloring the first vertex $1$, we can choose any $m_1$ colors from the
given $q$ colors and color the vertex $1.$ To color the vertex $2$, we can not use any colors that are already used to color the vertex $1$, and this is the only restriction, so we can choose any $m_2$ colors from the remaining $q-m_1$ colors and color the vertex $2$. Similarly, we color all the remaining vertices as they are adjacent to only one vertex, namely vertex $1.$
Now, since both $\pi_{S_n}^{\mathbf{m}}(q)$ and $\binom{q}{m_1}\binom{q-m_1}{m_2}\cdots \binom{q-m_{1}}{m_n}$ are polynomials in $q$ and they coincide for all natural numbers $q$, we must have the equality as polynomials, in other words, \[
\pi_{S_n}^{\mathbf{m}}(q) = \binom{q}{m_1}\binom{q-m_1}{m_2}\cdots \binom{q-m_{1}}{m_n}. 
\] holds for $q\in \mathbb{C}.$
\begin{figure}[h]
\centering
\begin{subfigure}{0.49\textwidth}
\centering
\begin{tikzpicture}
        \tikzstyle{B}=[circle,draw=black!80,fill=black!80,thick]
        \node[B] (1) at (0,0) [label=below:$n-2$]{};
        \node[B] (2) at (1.5,0) [label=below:$4$]{};
        \node[B] (3) at (-0.8,1.5) [label=left:$n-1$]{};
        \node[B] (4) at (2.3,1.5) [label=right:$3$]{};
        \node[B] (5) at (1.5,3) [label=right:$2$]{};
        \node[B] (6) at (0.75,1.5) [label=below:$1$]{};
        \node[B] (7) at (0,3) [label=left:$n$]{};
        \path[-] (1) edge node[left]{} (6);
        \path[-] (2) edge node[left]{} (6);
        \path[-] (3) edge node[left]{} (6);
        \path[-] (4) edge node[left]{} (6);
        \path[-] (5) edge node[left]{} (6);
        \path[-] (7) edge node[left]{} (6);
        \path (1) -- node[auto=false]{\ldots} (2);
    \end{tikzpicture}
    \caption{Star graph $S_n$}
    \label{stargraph}
\end{subfigure}
\caption{}
\end{figure}

\noindent
We have the following result for cycle graphs due to Ronald C. Read. Set $m_{n+1}=m_1$  and $v_k(q)=\binom{q}{k}-\binom{q}{k-1}$ and $(q)_k=q(q-1)(q-2)\cdots(q-k+1)$, then we have
\begin{thm}(\cite[Theorem 3]{Read})\label{cyclegraphread}
For the cycle graph $C_n$, $n\ge 3$, we have the following formula
\[  \pi_{C_n}^{\mathbf{m}}(q)= \frac{1}{\prod_{i=1}^n m_i!}  \left(\prod_{i=1}^{n}(q)_{(m_i+m_{i+1})}\right)\times
    \left(\sum\limits_{k=0}^n\Big((-1)^{kn}v_k(q)\prod_{i=1}^n\frac{(m_i)_{(k)}}{(q)_{({m_i+k})}}\Big)\right)
\]\qed
\end{thm}

\subsection{Perfect Elimination Ordering and Chordal graphs}
 
A graph $\mathcal{G}$ is called \textit{chordal} if it has no induced subgraph isomorphic to the cycle graph $C_m$ with $m \geq 4$.   
Computing the chromatic polynomials for finite chordal graphs is easy as they have perfect elimination ordering; see \cite{Gavril}. We now recall the definition of 
perfect elimination ordering of vertices. First, denote by $N_\mathcal{G}[k]$ the set of all neighbors of $k$ in $\mathcal{G}$ including $k$, i.e.,
$$N_\mathcal{G}[k]=\{ r\in I\backslash \{k\} : e(r, k)\in E\}\cup \{k\}.$$
\begin{defn} Now we define the perfect elimination ordering and PEO graphs. 
\begin{enumerate}
    \item The ordering on $I$ of the vertices of $\mathcal{G}$ is called 
a \textit{perfect elimination ordering} on $I$ if it is a well-ordering of $I$ and 
for each $k\in I$ the subgraph $\mathcal{G}_k$ of $\mathcal{G}$ induced by the set of vertices 
$N_\mathcal{G}[k]\cap \{r\in I : r\le k\}$ is a clique, i.e., a complete subgraph of $\mathcal{G}$. 

\medskip
\item A graph $\mathcal{G}$ is said to be \textit{PEO--graph} if it has an ordering of its vertices that is a perfect elimination ordering.

\end{enumerate}
\end{defn}
\noindent
We identify $I$ with $\{1, \ldots, n\}$ if its cardinality is $n$. We assume that this ordering of the vertices is a perfect elimination ordering if $\mathcal{G}$ is a PEO--graph. 
The following result is well-known for finite simple graphs \cite[Theorem 4.1]{Golumbic}, see also \cite{Fulkerson}.
\begin{prop}
A finite simple graph $\mathcal{G}$ is chordal if and only if it has a perfect elimination ordering. \qed
\end{prop}
    
\medskip

\subsection{}We have explicit formulas for the generalized chromatic polynomials of graphs that have perfect elimination ordering.\begin{thm}\label{poegraph}
Let $\mathcal{G}$ be a $PEO-$graph with a vertex set $I$. Let $I = \{1, \dots, n\}$ be a perfect elimination ordering on $I.$
For $\mathbf{m}\ge 0$ and $q\ge 1$, we have
 \[
\pi_{\mathcal{G}}^{\mathbf{m}}(-q) =(-1)^{|\bold m|}\prod\limits_{r=1}^{n} \binom{q-1+a_r(\mathbf{m})}{m_{r}}
\]
where $a_r(\mathbf{m})=\sum\limits_{\substack{1\le s \le r\\ s\in N_\mathcal{G}[r]}} m_{s}$
for $1\le r \le n$. \end{thm}

\begin{pf}
Since $m_i = 0$ if $i\notin \mathrm{supp}(\mathbf{m})$, we can restrict our attention to the induced subgraph $\mathcal{G}(I_N)$,
where $I_N= \mathrm{supp}(\mathbf{m})$. Write $ \mathrm{supp}(\mathbf{m}) = \{i_1<\cdots < i_N\}$,
note that this induced ordering is also a perfect elimination ordering for $\mathcal{G}(I_N)$. 
If we start coloring the vertices of $\mathcal{G}(I_N)$, the first vertex $i_1$ has no restrictions, so we can choose $m_{i_1}$ colors from the given $q$ colors and color it. 
By induction, let us say we have colored the first $r-1$ vertices, $r\ge 2$. 
For $2\le r<N$, the subgraph spanned by $N_{\mathcal{G}(I_N)}[i_r]\cap \{i_1, \dots, i_r\}$ is a clique. So, the only restriction that we have to color the vertex $i_r$ is that
the colors that are used to color the vertices  $N_{\mathcal{G}(I_N)}[i_r]\cap \{i_1, \dots, i_{r-1}\}$ should not be used to color the vertex $i_r$.
With this restriction, we can choose 
any $m_{i_r}$ colors from the remaining $q-a_r(\mathbf{m})+m_{i_r}$ colors and color the vertex $i_r.$ Thus, for any $q\in \mathbb{N}$, we have 
\begin{equation}\label{equproof}
    \pi_{\mathcal{G}}^{\mathbf{m}}(q) =\prod\limits_{r=1}^{N} \binom{q-a_r(\mathbf{m})+m_{i_r}}{m_{i_r}}
\end{equation}
Since both sides of the equation \ref{equproof} are polynomials in $q$, and they agree for all positive integers $q$, both polynomials must coincide for all $q\in \mathbb{C}.$ In particular, we have \[
    \pi_{\mathcal{G}}^{\mathbf{m}}(-q) =\prod\limits_{r=1}^{N} \binom{-q-a_r(\mathbf{m})+m_{i_r}}{m_{i_r}}=\prod\limits_{r=1}^{N} (-1)^{m_{i_r}}\binom{q+a_r(\mathbf{m})-1}{m_{i_r}}
\] for any $q\ge 1.$
This implies that \[
    \pi_{\mathcal{G}}^{\mathbf{m}}(-q) = (-1)^{|\bold m|}\prod\limits_{r=1}^{N} \binom{q-1+a_r(\mathbf{m})}{m_{i_r}}
\] for all $q\ge 1.$ This completes the proof as $m_i = 0$ if $i\notin \mathrm{supp}(\mathbf{m})$.
\end{pf}
\noindent
We see some examples now.\medskip
\begin{enumerate}\label{keyexamplechordal}
    \item The ordering of the vertices of the line graph $L_n$ given in Figure \ref{L_n} is a perfect elimination ordering. Thus we have
 \[
\pi_{L_n}^{\mathbf{m}}(q) = \binom{q}{m_1}\binom{q-m_1}{m_2}\cdots \binom{q-m_{n-1}}{m_n}, \, \text{and}
\] 
\[
\pi_{L_n}^{\mathbf{m}}(-q) =(-1)^{|\mathbf{m}|} \binom{q+m_1-1}{m_1}\binom{q+m_1+m_2-1}{m_2}\cdots \binom{q+m_{n-1}+m_n-1}{m_n}.
\]

\medskip
\item Recall that the complete graph $K_n$ is a graph on vertices $\{1, \ldots, n\}$ such that any two distinct vertices $i, j$ 
has an edge between them. For example, the complete graph on $4$ vertices is given below:
\begin{figure}[h]
\centering
\begin{subfigure}{0.49\textwidth}
\centering
    \begin{tikzpicture}
    \node (1) at (0,0) [circle, fill=black,thick, label=left:1] {};
    \node (2) at (4,0) [circle, fill=black, thick, label=right:2] {};
    \node (3) at (2,3.732) [circle, fill=black, thick, label=above:3] {};
    \node (4) at (2,1) [circle, fill=black, thick, label=below:4] {};
    \path[-] (1) edge node[left]{} (2);
    \path[-] (1) edge node[left]{} (3);
    \path[-] (1) edge node[left]{} (4);
    \path[-] (2) edge node[left]{} (3);
    \path[-] (2) edge node[left]{} (4);
    \path[-] (3) edge node[left]{} (4);
\end{tikzpicture}
\caption{Complete graph $K_4$}
    \label{K4}
\end{subfigure}
\caption{}
\end{figure}

\medskip
\noindent
The ordering of the vertices $\{1,\ldots, n\}$ of $K_n$ is a perfect elimination ordering. So, we get
\[
\pi_{K_n}^{\mathbf{m}}(-q) =(-1)^{|\mathbf{m}|} \binom{q+m_1-1}{m_1}\binom{q+m_1+m_2-1}{m_2}\cdots \binom{q+m_1+\cdots +m_{n-1}+m_n-1}{m_n}.
\]
\end{enumerate}

\medskip
\medskip
\subsection{}\label{1implies2}
Now we record the following significant result from \cite[Propostion 3.7]{ADV2018} that connects the $q$--th power of the multi-variate independence polynomial of $\mathcal{G}$ and the generalized chromatic polynomials of $\mathcal{G}.$ We give proof of this fact here for the reader's convenience. 
 We set $\mathbf{x}^{\mathbf{m}}=\prod\limits_{i\in I}x_i^{m_i}$ for $\mathbf{m}\in \mathbb{Z}^I_+$ 
 and $\mathbf{x}^S=\prod\limits_{i\in S}x_i$ for $S\subseteq I$.
 For $f\in R$, we denote by $f[\mathbf{x}^{\mathbf{m}}]$  the coefficient of $\mathbf{x}^{\mathbf{m}}$ in $f.$
 
\begin{thm}(\cite[Propostion 3.7]{ADV2018})\label{keyresult}
Let $\mathcal{G}$ be a finite simple graph. For $q\in \mathbb{Z}$, we have the following identity holds in $R$
\[ I(\mathcal{G}, \mathbf{x})^{q} = \sum\limits_{\mathbf{m}\ge 0}\pi_{\mathcal{G}}^{\mathbf{m}}(q) \mathbf{x}^{\mathbf{m}}.
\]
\end{thm}
\begin{pf}

We expand 
$I(\mathcal{G}, \mathbf{x})^{q} $ using the binomial expansion. For any $f\in R$ with constant term one, we have
 $$f^q=\sum\limits_{k\ge 0}{q \choose k}\, (f-1)^k.$$
 Take $f= I(\mathcal{G}, \mathbf{x})$ and note that 
 $$f-1 = \sum_{ S \in \mathcal{I}(\mathcal{G}) \backslash \{\emptyset\} } \mathbf{x}^S.$$
 For $\mathbf{m}\in \mathbb{Z}_{\ge 0}^I$, the coefficient
  $(f-1)^k[\mathbf{x}^{\mathbf{m}}]$ is given by
$$\sum\limits_{(S_1,\dots,S_k)}1$$
where the sum ranges over all $k$--tuples $(S_1,\dots,S_k)$, with some possible repetitions, such that
\medskip
\begin{enumerate}
    \item  $S_i \in \mathcal{I}(\mathcal{G})\backslash \{\emptyset\}$ for each $1\le i\le k$, and
    \medskip
    \item the disjoint union $S_1\dot{\cup} \cdots \dot{\cup}S_k = \dot{\bigcup}_{i\in \mathrm{supp}(\mathbf{m})} \{ i, \ldots, i\}$, where $i$ appears exactly $m_i$ number of times for each $i\in  \mathrm{supp}(\mathbf{m})$.
\end{enumerate}

\medskip
\noindent
It follows that $\big(S_1.\dots, S_k)\in P_k\big(\mathbf m, \mathcal{G}\big)$ and each element is obtained in this way. 
So, the sum ranges over all elements in $P_k\big(\mathbf m, \mathcal{G}).$
Hence we have $$(f-1)^k[\mathbf{x}^{\mathbf{m}}] = |P_k\big(\mathbf m, \mathcal{G})|.$$ This implies that
\[ I(\mathcal{G}, \mathbf{x})^{q}[\mathbf{x}^{\mathbf{m}}]  = \sum\limits_{k\ge 0}{q \choose k}\, (f-1)^k[\mathbf{x}^{\mathbf{m}}] =  \sum_{k \geq 0} \binom{q}{k} |P_k(\mathbf{m}, \mathcal{G})|  =  \pi_{\mathcal{G}}^{\mathbf{m}}(q).
\]
\end{pf}
\noindent
We quickly recover the results \cite[Corollary 3.2 \& 6.2]{RadchenkoVilleas2021}, and they proved the following result for chordal graphs by entirely different techniques.
\begin{cor}\label{keyresultpoegraph}
Let $\mathcal{G}$ be a chordal graph with the vertex set $I$ and let $I=\{1, \ldots, n\}$ be a perfect elimination ordering. 
Then, for any $q\ge 1$,  we have
 \[  I(\mathcal{G}, \mathbf{x})^{-q} = \sum\limits_{\mathbf{m}\ge 0} (-1)^{|\mathbf{m}|}\prod\limits_{r=1}^{n} \binom{q-1+a_r(\mathbf{m})}{m_{r}} \mathbf{x}^{\mathbf{m}}.
\]
where $a_r(\mathbf{m})=\sum\limits_{\substack{1\le s \le r\\ s\in N_\mathcal{G}[r]}} m_{s}$
for $1\le r \le n$. In particular, the series $ I(\mathcal{G}, \mathbf{x})^{-q}$ is a Horn hypergeometric series for $q\ge 1.$
\end{cor}
\begin{proof}
Note that $a_r(\mathbf{m})\ge m_r$ for all $r\in I$. Write $ I(\mathcal{G}, \mathbf{x})^{-q} = \sum_{\mathbf{m}\ge 0}c_{\mathbf{m}} \bold x^{\mathbf{m}}$ and
$c_{\mathbf{m}, r} = \binom{q-1+a_r(\mathbf{m})}{m_{r}}$
We have
$$c_{\mathbf{m}, r} = \frac{(q+a_r(\mathbf{m})-1)\cdots (q+a_r(\mathbf{m})-m_r)}{m_r!}.$$ 
For $1\le i\le n$, we have $c_{\mathbf{m+e_i}}=(-1)^{|\mathbf{m+e_i}|}\prod_{r=1}^n c_{\mathbf{m+e_i},r}.$ Now note that we have
    $$c_{\mathbf{{m+e_i}}, r}=\begin{cases}
        \binom{q+a_r(\mathbf{m})}{m_r} & \text{ if } i<r\text{ and } (i,r)\in E \\
\noalign{\vskip9pt}
        \binom{q+a_i(\mathbf{m})}{m_i+1} & \text{ if } i=r\\
        c_{\mathbf{{m}}, r} & \text{ else.} 
    \end{cases}$$
and
    $$\frac{c_{\mathbf{{m+e_i}},r }}{c_{\mathbf{{m}}, r}} =\begin{cases}
       \frac{q+a_r(\mathbf{m})}{q+a_r(\mathbf{m})-m_r} & \text{ if } i<r\text{ and } (i,r)\in E \\ 
       \noalign{\vskip9pt}
       \frac{q+a_i(\mathbf{m})}{m_i+1} & \text{ if } i=r  \\  \end{cases}$$
    Therefore $$\frac{C_{\mathbf{m+e_i}}}{C_{\mathbf{m}}}=(-1)\left(\frac{q+a_i(\mathbf{m})}{m_i+1}\right)\left(\prod\limits_{\substack{r=i+1\\ (i,r)\in E}}^n \frac{q+a_r(\mathbf{m})}{q+a_r(\mathbf{m})-m_r}\right),$$ which is clearly
     a rational function in the variables $(m_j : j\in I)$ and this completes the proof.
\end{proof}
\medskip
\begin{cor}(\cite[Theorem 3]{Read})\label{keyresultcycle}
For the cycle graph $C_n$, $n\ge 3$ and $q\ge 1$ we have the following formula
 \[
 I(C_n, \mathbf{x})^{-q}[\mathbf{x}^{\mathbf{m}}] =  \pi_{C_n}^{\mathbf{m}}(-q)= \frac{1}{\prod_{i=1}^n m_i!}  \left(\prod_{i=1}^{n}(-q)_{(m_i+m_{i+1})}\right)\times
    \left(\sum\limits_{k=0}^n\Big((-1)^{kn}v_k(-q)\prod_{i=1}^n\frac{(m_i)_{(k)}}{(-q)_{({m_i+k})}}\Big)\right).
\]  In particular, we have the following formula
for $q=1$ and $\mathbf{a} = (a,\dots,a)$
$$I(C_n, \mathbf{x})^{-1}[\mathbf{x}^\mathbf{a}] = (-1)^{na}\left(\sum_{|k|\leq a}(-1)^k\binom{2a}{a+k}^n\right).$$
\end{cor}

\begin{proof}
The first formula follows from the Proposition \ref{keyprop} and Read's formula Theorem \ref{cyclegraphread}. 
    It remains to consider the case $q = -1$ and $\mathbf{a} = (a,\dots,a)$. Recall that $\binom{-1}{k} = (-1)^k$ and $v_k(-1) = \binom{-1}{k} - \binom{-1}{k-1} = 2 (-1)^k.$
    We have
    \begin{align*}
       (a!)^n I(C_n, \mathbf{x})^{-1}[\mathbf{x}^\mathbf{a}] &= ((-1)_{(2a)})^n\left(\sum_{k=0}^{n}(-1)^{kn} v_k(-1)\Big(\frac{(a)_{(k)}}{(-1)_{(a+k)}}\Big)^n\right) \\
        &=((-1)_{(2a)})^n\left(\sum_{k=1}^{n} 2\cdot(-1)^{k(n+1)}\Big(\frac{(a)_{(k)}}{(-1)_{(a+k)}}\Big)^n\right) + ((-1)_{(2a)})^n \Big(\frac{(a)_{(0)}}{(-1)_{(a)}}\Big)^n\\
 &= ((2a)!)^n\left(2\left(\sum_{k=1}^{n}(-1)^{na+k}\Big(\frac{a!}{(a-k)!(a+k)!}\Big)^n\right)+(-1)^{na}\Big(\frac{1}{a!}\Big)^n\right)\\
        &= \frac{(a!)^n}{(-1)^{na}}\left(2\sum_{k=1}^a(-1)^k\binom{2a}{a+k}^n + \binom{2a}{a}^n\right)\\\end{align*}
Here we use $(-1)_{(m)} = (-1)^m m!$.
This implies that \[
     I(C_n, \mathbf{x})^{-1}[\mathbf{x}^\mathbf{a}]  = (-1)^{na}\left(2\sum_{k=1}^a(-1)^k\binom{2a}{a+k}^n + \binom{2a}{a}^n\right),\]
    which is our required formula. 
\iffalse 
This implies
 \begin{align*}
  I(C_n, \mathbf{x})^{-1}[\mathbf{x}^\mathbf{a}] &= \frac{(a!)^n}{(-1)^{na}}\left(\sum_{|k|\leq a}(-1)^k\binom{2a}{a+k}^n\right)
    \end{align*}\fi
\end{proof}

\subsection{} 
Let $S\subseteq I$ be a subset of $I.$ Consider the following natural ring homomorphism $$\pi_S: \mathbb{Q}[[x_i : i\in I]]\to  \mathbb{Q}[[x_i : i\in S]]$$
given by $\pi_S(F(\bold x))=\sum\limits_{\substack{\mathbf{m}\ge 0\\ \mathrm{supp}(\mathbf{m})\subseteq S }} c_{\mathbf{m}} \bold x^{\mathbf{m}}$, where
 $F(\bold x) = \sum\limits_{\mathbf{m}\ge 0} c_{\mathbf{m}} \bold x^{\mathbf{m}}\in \mathbb{Q}[[x_i : i\in I]].$
That means $\pi_S$ is defined by the specializing  $x_i = 0$ for all $i\notin S.$ The following local property of Horn hypergeometric series is easy to prove.
\begin{prop}\label{localproperty}
Suppose $F(\bold x)\in \mathbb{Q}[[x_i : i\in I]]$ is a Horn hypergeometric series then $\pi_S(F(\bold x))\in \mathbb{Q}[[x_i : i\in S]]$ is 
Horn hypergeometric for all $S\subseteq I$. \qed
\end{prop}
It is easy to see that the graph operation taking induced subgraphs correspond to these specializations for multi-variate independence polynomials.  
More precisely, we have 
\begin{prop}\label{restind}
Let $\mathcal{G}$ be a finite simple graph with the vertex set $I$ and let $S\subseteq I$. Then we have
$\pi_S(I(\mathcal{G}, \bold x)) = I(\mathcal{G}(S), \bold x)$ and $\pi_S(I(\mathcal{G}, \bold x)^{-1}) = I(\mathcal{G}(S), \bold x)^{-1}$. 
\end{prop}
\begin{proof}
We have $\mathcal{I}(\mathcal{G}(S)) = \{ J\in \mathcal{I}(\mathcal{G}) : J\subseteq S\}$, now the result follows from the definition of multi-variate independence polynomials
and $\pi_S$ being ring homomorphism.
\end{proof}
\noindent
We get the following important result by combining the Propositions \ref{localproperty} and \ref{restind}. 
\begin{cor}\label{corhyper}
Let $\mathcal{G}$ be a finite simple graph with the vertex set $I$. 
Suppose $I(\mathcal{G}, \bold x)^{-1}$ is Horn hypergeometric then $I(\mathcal{G}(S), \bold x)^{-1}$ is Horn hypergeometric for all $S\subseteq I$. \qed
\end{cor}

\medskip
\subsection{ Proof of Theorem \ref{mainthmintro}}\label{converseproof}
Only the implication $(3) \implies (1)$ remains.  Assume that the power series $I(\mathcal{G}, \mathbf x)^{-1}$  is Horn hypergeometric, but
$\mathcal{G}$ is not a chordal graph. Then it has an induced subgraph $C_m$ for some $m\ge 4.$ This would imply that $I(C_m, \mathbf x)^{-1}$  is Horn hypergeometric by corollary \ref{corhyper}. Write  $I(C_m, \mathbf x)^{-1} = \sum\limits_{\mathbf{m}\ge 0} c_{\mathbf{m}} \bold x^{\mathbf{m}}$. Consider the  
 sum of (signed) diagonal terms $$H_m(t) = \sum\limits_{\mathbf{a}= (a,\ldots, a)} (-1)^{ma} c_{\bold a} t^{a}\in \mathbb{Q}[[t]].$$
 Then the series $H_m(t)$ is Horn hypergeometric since 
 $$\frac{c_{\bold a+ e_1+\cdots +e_m}}{c_{\bold a}} = 
 \frac{c_{\bold a + e_1}}{c_{\bold a}}\frac{c_{\bold a + e_1+e_2}}{c_{\bold a+e_1}}\cdots \frac{c_{\bold a + e_1+\cdots + e_{m}}}{c_{\bold a+e_1+\cdots+e_{m-1}}}.$$
 is a rational function in the variable $a$ as the terms on the right side of the product are rational functions in the variable $a$.
 But we know from Corollary \ref{keyresultcycle} that, $$(-1)^{ma} c_{\bold a} = \left(\sum_{|k|\leq a}(-1)^k\binom{2a}{a+k}^m\right) = : S(m, a).$$ Now the proof follows from the argument given in  \cite[Proposition 6.3]{RadchenkoVilleas2021}. We repeat it here for the reader's convenience.
The numbers $S(m, a)$ are known as de Bruijn numbers in the literature; see \cite[Page 72, Section 4.7]{Bruijn}. 
The asymptotic behavior of $S(m, a)$ for fixed $m$ and large $a$ has been computed in \cite{Bruijn}:
	$$\frac{S(m, a+1)}{S(m, a)}\rightarrow (2\cos(\pi/2m))^{2m}, \qquad a \rightarrow\infty.$$
Since $H_m(t)$ is Horn
hypergeometric, for a fixed $m$, we can write the ratio $\frac{S(m, a+1)}{S(m, a)} = \frac{p(a)}{q(a)}$ where $p(x), q(x) \in \mathbb{Q}[x]$ are polynomials over $\mathbb{Q}.$
Write $p(x) = b_0+b_1x+\cdots +b_rx^r$ and $q(x) = c_0+c_1x+\cdots +c_sx^s$ with $b_r\neq 0$ and $c_s\neq 0$. Then we have 
$$\frac{S(m, a+1)}{S(m, a)} = a^{r-s}\left(\frac{b_0/a^r+\cdots +b_{r-1}/a +b_r}{c_0/a^s+\cdots +c_{s-1}/a +c_s}\right) \rightarrow (2\cos(\pi/2m))^{2m}, \qquad a \rightarrow\infty.$$
Hence we must have $r=s$ and  $(2\cos(\pi/2m))^{2m} = \frac{b_r}{c_r}\in \mathbb{Q}$, and 
hence $m\le 3$. This leads to a contradiction. Thus, our assumption that $\mathcal{G}$ is not a chordal graph is wrong, and this completes the proof. 

\medskip\medskip
\noindent\textbf{Acknowledgment.} We thank Bernd Sing from the Department of Mathematics, University of the West Indies, Cave Hill, Barbados, for providing a copy of Ronald C. Read's article. I.H. was partially supported by DST/INSPIRE/03/2019/000172.

\medskip
\noindent\textbf{Data availability.} No data is associated.

\medskip\medskip
\noindent\textbf{Declarations.}

\medskip
\noindent\textbf{Conflict of Interest.} On behalf of all authors, the corresponding author states there is no conflict of interest.

\end{document}